\documentclass[11pt, reqno]{amsart}

\usepackage[utf8]{inputenc}
\usepackage[T2A]{fontenc}
\usepackage[ukrainian,russian,english]{babel}

\usepackage[english]{d-omega3}

\usepackage{amssymb}
\usepackage{amscd}
\usepackage{graphicx}
\usepackage{xcolor}
\usepackage[all]{xy}

\DeclareMathOperator{\arccot}{arccot}
\DeclareMathOperator{\arccoth}{arcoth}
\newcommand{\Int}{\text {int} \ }

\newcommand{\F}{\mathcal F} 

\newcommand{\R}{\mathbb R} 
\newcommand{\Z}{\mathbb Z}

\begin{document}
	
	\author[D.~  Bolotov]{Dmitry V. Bolotov}
	\email{bolotov@ilt.kharkov.ua}
	\address{B. Verkin Institute for Low Temperature Physics and Engineering of the
		National Academy of Sciences of Ukraine, 47 Nauky Ave., Kharkiv, 61103, Ukraine}
		\orcid{0000-0002-8542-9695}
		\title[Dual Thurson norm of  Euler classes of foliations on 3-Manifolds]{ Dual Thurston norm of    Euler classes of foliations   on negative curvature 3-Manifolds}
		
	\abstract{english}{%
 In this paper we give an upper bound estimate on the dual Thurston norm of the Euler class of an arbitrary smooth foliation $\F$ of dimension one defined on a closed three-dimensional orientable manifold $M^3$ of negative curvature, which depends on the constants bounded the injectivity radius $inj(M^3)$, the volume $Vol(M^3)$,  sectional curvature of the manifold $M^3$ and the mean curvature modulus of the leaves of the foliation $\F$. 
	}
	
	\shortAuthorsList{D.Bolotov}

	\maketitle

	\section{Introduction}\label{Int}
	
	Let $ (M^3, g) $ be a closed oriented three-dimensional Riemannian manifold and  $ \F $ be a transversely oriented $ C ^{\infty} $-smooth foliation of codimension one on $ M^3 $. 
	Recall that a foliation $\F $ is {\it taut}  if its leaves are minimal submanifolds of $ M^3 $ for some Riemannian metric on $ M^3 $. D. Sullivan \cite{Sul} proved that a foliation is taut if and only if each leaf of $\F$ is intersected by a transversal closed curve, which   is equivalent to $\F$  does not contain generalized Reeb components.  	Recall, that  a saturated (consisting of the leaves of $\F$) set $ A $  of a three-dimensional compact orientable foliated 3-manifold $ M^3 $ is called a {\it generalized Reeb component} if $ A $ is a connected three-dimensional manifold with a boundary $\partial A$ and any transversal to $ \F $ vector field restricted to   $\partial A $ is directed either everywhere inwards or everywhere outwards of the generalized Reeb component $ A $. In particular, the Reeb component (see \cite {T}), i.e. foliated solid torus all leaves  of which except for the boundary  are homeomorphic  to the plane (see Fig.\ref{Reeb3}),  is a generalized Reeb component. It is clear that $ \partial {A} $ consists of a finite set of compact leaves of the foliation $ \F $.  It is not difficult to show that $ \partial A $ is a family of tori (see \cite {Good}).

	Recall that 3-manifold $M^3$ is called {\it irreducible} if each an embedded sphere bounds  a ball in $M^3$. In particular,   $\pi_2(M^3)=0$.   Note,  that   if $M^3$ admits a taut foliation, then $M^3$ is irreducible \cite{Nov}.
	
	W. Thurston has proved  \cite{Th}  that for each    closed embedded orientable surface $M^2\subset M^3$ which is different from $S^2$  the value of the Euler class $e(T\F)$ on the class $[M^2]$ in the case of a taut $\F$ satisfies  the following estimate:
	
	\begin{equation}\label{Chi}
		|e(T\F)[M^2]|\leq - \chi (M^2). 
	\end{equation}
	
	Since any integer homology class $H_2(M^3;\Z)$ can be represented by a closed oriented surface (see subsection \ref{harm}) the inequality above gives a bound for the possible values of the cohomology class $e(T\F)$ on the generators of $H_2(M^3;\Z)$ and  therefore  the number of cohomological classes $ H^2 (M^3;\Z) $ realized as Euler classes $e(T\F)$ of the tangent distribution to   $ \mathcal F $ is finite. 
 The inequality \eqref{Chi} means that 
	the dual Thurson norm  (see subsection \ref{3.2.}) of the Euler class $e(T\F)$ of a taut foliation  $ \mathcal F $  on a  closed oriented  irreducible  atoroidal  (containing no incompressible tori) 3-manifold  $M^3$  satisfies:
		\begin{equation}\label{Chi1} 
			||e(T\F)||^*_{Th}\leq 1. 
		\end{equation}		
	
	\begin{remark}
		A 3-manifold $M^3$ of negative sectional curvature is an example of an irreducible  atoroidal 3-manifold.
		\end{remark}
		
		\begin{remark}
	Note, that taut foliations play an important role in three-dimensional Seiberg-Witten theory. Namely, the Euler class $e(T\F)$  of a such foliation  can be interpreted as a monopole class  \cite{KM}.
	
	\end{remark}

	 In  \cite {B1} we  proved the following result.
	
	\begin{theorem} \label{result1}
		
		Let $ V_0> 0, i_0> 0, K_0 \geq 0 $ be fixed constants, and $ M^3 $ be a closed oriented three-dimensional Riemannian manifold with the following properties:
		\begin{enumerate}
			\item the volume $Vol(M^3)\leq V_0$; 
			\item the sectional curvature $K$  of  $M$  satisfies the  inequality $K  \leq K_0$;
			\item  $\min \{inj(M^3),\frac{\pi}{2\sqrt{K_0}}\} \geq i_0 $, where $ inj (M^3) $ is the injectivity   radius of  of  $ M^3 $.
		\end{enumerate}
		Let us set
		\[ H_0 = \begin{cases}
			\min \{{\frac {2\sqrt{3}i_0^2}{V_0},\sqrt[3]{\frac{2\sqrt{3}}{V_0}}}\},     & \text{if $K_0 = 0$,}\\
			\min \{ {\frac {2\sqrt{3}i_0^2}{V_0}, x_0 } \},  &  \text{if $K_0 >0$,}
		\end{cases} \]
		where $x_0$ is the root of the equation
		$$ \frac{1}{ K_0 }\arccot^2\frac{x}{\sqrt{K_0}} - \frac{V_0}{2\sqrt{3}}x=0.  $$
		Then any smooth transversely oriented foliation $ \F $ of codimension  one on $ M^3 $, such that the modulus of the mean curvature $H$ of its leaves  satisfies the inequality $ | H | <H_0 $,
		should be taut, in particular, have minimal leaves for some Riemannian metric on $ M^3 $.
		
	\end{theorem}
	
	If we take  the constant    $H_0$  to be    an arbitrary positive real  value bounding  from above  the modulus of the mean curvature $H$ of the leaves:     $ | H | <H_0 $,    then the foliation $\F$ may already contains (generalized) Reeb components.  Assuming that the constant $K_0$ takes also negative values, the number of Reeb components is estimated as follows (see \cite{B2}, \cite{B11}):
	$$
	 \text {\{the number of  Reeb components of $\F$\}} \leq    \frac{4 H_0V_0} {\sqrt{3}C^2_0},  
	$$
	 where
		\begin{equation}\label{C_0}			
		C_0=  \left \{\begin{array}{ll}
			2 \min \{{i_0}, \frac{1}{ \sqrt{K_0}} \arccot\frac{H_0}{\sqrt{K_0}}\}  , &\text {if } K_0>0 \\
			2 \min \{i_0, \frac{1}{H_0}\}, &\text {if } K_0=0 \\
			2 \min \{{i_0}, \frac{1}{ \sqrt{-K_0}} \arccoth\frac{H_0}{\sqrt{-K_0}}\}  , &\text {if } K_0<0 \ \& \  H_0 >\sqrt{-K_0} \\
			2  {i_0}, &\text {if } K_0<0 \ \& \  H_0 \leq \sqrt{-K_0}

		\end{array}  \right.
	\end{equation}

The constant $C_0$ has the following geometrical sense:
	$$ sys(\F)\geq C_0,$$
	where $sys (\F)$ (foliated systola) is the length  of the shortest   {\it integral loop} (i.e. a loop belonging to  some leaf of the foliation $\F$) among all  leafwise non-contractible  ({\it essential}) integral loops.  Such a loop  exists and is  geodesic in the  leaf (see (\cite{B11}).   
	
\begin{remark}
It follows from both Reeb stability theorem  (see \cite{R}) and Rosenberg's   theorem  (see \cite{Ros}) that  excepting for the two cases below, the foliation $\F$ contains a not simply connected leaf. 
	
	\begin{itemize}
		\item  $\F$ is a fibration  of  $S^2\times S^1$ by spheres $S^2\times *$.  Clearly, in this case   $M^3$ is not irreducible;
		\item  $\F $ is a foliation by planes. In this case $M^3\cong T^3$, in particular, $M^3$ is not atoroidal.  
		\end{itemize}
	
	  \end{remark}

		In our recent work \cite{B2}, we gave an upper bound on the $L^2$ - norm of the real Euler class $e(T\F)$ of an arbitrary transversally oriented foliation $\F$ of codimension one, defined on a three-dimensional closed irreducible orientable Riemannian 3-manifold $M^3$.  We proved:

	\begin{theorem}\label{main}
		Let $ V_0> 0, i_0> 0, H_0>0, k_0 \leq K_0$, where $k_0\leq0$,  be fixed constants.  Suppose   $ (M^3,\F) $ is a closed oriented irreducible three-dimensional Riemannian manifold equipped by  a two-dimensional transversely oriented foliation $\F $, whose leaves  have the modulus of mean curvature $H$ bounded above by the constant $H_0$,  and $M^3$  satisfies the  following conditions:
		\begin{enumerate}
			\item the volume $Vol(M^3)\leq V_0$; 
			\item the sectional curvature $K$  of  $M$  satisfies the  inequality $k_0\leq K  \leq K_0$;
			\item  $$\left\{\begin{array}{lr} \min \{inj(M^3),\frac{\pi}{2\sqrt{K_0}}\}, &\text {if } K_0>0\\
				inj(M^3)&\text {if } K_0\leq0 	\end{array}  \right\}\geq i_0,$$ where $ inj (M^3) $ is the injectivity   radius of  of  $ M^3 $.
		\end{enumerate}
		Then  there exist constants $C_0 (i_0, H_0, K_0)$ and $\Lambda(k_0,V_0,i_0)$  such that 
	
		\begin{equation}\label{euler1}
		||e(T\F)||_{L^{2}}  \leq -\frac{3}{\pi}k_0\sqrt{V_0}  +   \frac{32 H^2_0V_0^{\frac{3}{2}}} {3C^3_0}{\Lambda},
	\end{equation}

	\end{theorem}	
	
	\begin{remark}
			The constant   $C_0 (i_0, H_0, K_0)$ is defined in \eqref{C_0}, 		and  $\Lambda(k_0,V_0,i_0)$ is taken from the inequality  \eqref{Pet3}.
		\end{remark}

 In this paper  we present the following result.
	
	\begin{theorem}[Main theorem]\label{Main}
	Let  $ (M^3,\F) $ be a closed oriented  three-dimensional Riemannian manifold equipped by  a two-dimensional transversely oriented foliation $\F $, whose leaves  have the modulus of mean curvature $H$ bounded above by the constant $H_0\geq 0$,  and $M^3$  satisfies the  following conditions:
	
	\begin{enumerate}
		\item  $Vol(M^3)\leq V_0$;
		\item $k_0\leq K \leq K_0$;
		\item $inj(M^3)\geq i_0.$ 
	\end{enumerate}
	for some fixed constants $ V_0>0, i_0>0, H_0>0, k_0< K_0<0$ bounding  the volume $Vol(M^3)$, the sectional curvature $K$ of $M^3$ and the injectivity radius $inj(M^3)$.

 Then
		\begin{enumerate}
			\item [a)]  If $H_0\leq \sqrt {-K_0}$,  then   $\F$ is tout and $||e(T\F)||^*_{Th} \leq 1$.
			
			\item [b)] If $H_0> \sqrt {-K_0}$, then the dual Thurson norm  is estimated from above as follows:
			
			\begin{equation}\label{euler3}
				||e(T\F)||^*_{Th} \leq 6\frac{k_0}{K_0}\Lambda   -   \frac{64\pi H^2_0V_0} {3K_0C^3_0}\Lambda^2,
			\end{equation}
			where $\Lambda$  and $C_0$ are as in Theorem \ref {main}.
			\item [c)] If  $M^3$  is a hyperbolic form ($k_0=K_0\equiv -1$), then 
			$$
			||e(T\F)||^*_{Th}  \leq 30\frac{\sqrt{V_0}}{\sqrt{i_0}}  +  \frac{1600 \pi H^2_0V_0^2}{3C^3_0 \ {i_0}}.
			$$
		
		\end{enumerate}

	\end{theorem}
	\begin{corollary}\label{cormain}
	If $\F$ contains a compact leaf $\mathcal K$, then 
		$$
		\frac{|\chi(\mathcal K)|}{|| [\mathcal K] ||_{Th}} \leq 30\frac{\sqrt{V_0}}{\sqrt{i_0}}  +  \frac{1600 \pi H^2_0V_0^2}{3C^3_0 \ {i_0}}.
		$$
		
		\end{corollary}

	\section {Background material}

		\subsection {Comparison inequalities for  Mean curvature}
			\label{sec:1}
		
		Recall the following  comparison theorem for the normal curvatures.
		
		\begin{theorem}\cite [ 22.3.2.]{BZ}\label {BZ}
			Let $p\in M$ and $\ \beta:[0,r]\to M$ be a radial  geodesic  of the ball $B(p,r)$ of radius $r$ centered at the point $p$ of the Riemannian manifold $M$. Let   $\beta(r)$ be a point  not conjugate with $p$ along $\beta$. Let the radius $r$ be such that there are no conjugate points in the space of constant curvature $K_0$ within the radius of length $r$. Then if at each point $\beta (t)$ the sectional curvatures $K$ of the manifold $M$ do not exceed $K_0$, then the normal curvature $k^S_n$ of the sphere $S(p,r)$ at the point $\beta(r)$ with respect to the normal $-\beta'$ is not less than the normal curvature $k_n^0$ of the sphere of radius $r$ in the space of constant curvature $K_0$.
		\end{theorem}

		Let $M^3$ be a 3-Manifold satisfying the condition of  Theorem \ref{Main}.  Note that   all normal curvatures of the sphere $S(r)\subset M^3$ of radius $r$ are positive, provided that $r<i_0$ and the normal to the sphere $S(r)$ is directed inside the ball $B(r)$ which it bounds\footnote {The sphere $S(r)$ indeed bounds the ball, since by definition $r< inj (M^3)$.}. We will call such a normal  {\it inward}.

		\begin{definition}
			We  will call a hypersurface $S\subset M^3$ of the Riemannian manifold $M^3$    {\it supporting} to the subset $A\subset M^3$ at the point $p\in \partial A\cap S$ with respect to the normal $n_p \perp T_pS$, if $S$ cuts some spherical neighborhood $B_p$ of the point $p$ into two components, and $A\cap B_p$ is contained in that component to which the normal $n_p$ is directed.
			We will call the sphere $S(r)\subset M^3$ ($r< i_0$) the {\it supporting sphere} to the set $A\subset M^3$ at the point $q\in A\cap S(r)$ if it is the supporting sphere to $A $ at the point $q$ with respect to the inward normal.
		\end{definition}

		The following  lemma is obvious. 
		
		\begin{lemma}(\cite[Lemma 4]{B1})\label{l1}
			Assume that the sphere $S(r_0)$ ($r_0< i_0$) is the supporting sphere to the surface $F\subset M^3$ at the point $q $. Then $k^S_n(v) \leq k_n^F(v)\ \forall v\in T_qS(r_0)$, where $k^S_n(v)$ and $k^F_n(v)$ denote corresponding normal curvatures of $S(r_0)$ and $F$ at the point $q$ in the direction $v$.
		\end{lemma}
		
		As a consequence of   Lemma \ref{l1} and  Theorem \ref{BZ}  we obtain the following inequalities  at the touching point $q$ :
		\begin{equation}\label{ineq1}
			0< H_r^0\leq H_r(q) \leq H (q),
		\end{equation}
		where $H^0_r $ and $H_r$ are the mean curvatures of the spheres $S(r)$ bounding the ball of radius $r, \ r<i_0,$ in the space of constant curvature $K_0$ and the manifold $M^3$ respectively, and $H$ is the mean curvature of the surface $F$ .

		\subsection{Novikov's theorem and  vanishing cycle.} \label{2.3}

		Let $(M^3,\F)$ be a foliated closed  3- Manifold. An integral  loop   $\alpha:S^1\to M^3$ is a {\it vanishing cycle} if there exists a homotopy $A: S^1\times I \to M^3$ through integral loops $A_t:=A|_{S^1\times t}$ for $\F$ such that $A_0=\alpha$ and $A_t$ is inessential (i.e. is contractible in the supporting  leaf) for $0<t\leq 1$. 
		A vanishing cycle $\alpha $ is {\it non-trivial} if $\alpha$ is essential.
		
		The following well-known Novikov's theorem  is a topological criterion for the existence of a Reeb component in a foliation $\F$.

		\begin{theorem}  \cite {Nov}\label{Nov}. 
			\begin{enumerate}
				\item For a closed, orientable smooth 3-Manifold $M^3$ and a transversely orientable  $C^2$-smooth foliation $\F$ of codimension one on   $M^3$, the following are equivalent.
				\begin{enumerate}
					\item  The foliation $\F$  has a Reeb component.
					\item There is a leaf $\mathcal L$ of $\F$  that is not $\pi_1$-injective. That is, the inclusion $i: \mathcal L \to M^3$ induces a homomorphism $i_*:\pi_1({ \mathcal L})\to \pi_1(M^3)$  with nontrivial kernel. 
					\item Some leaf of $\F$  contains a non-trivial vanishing cycle.
				\end{enumerate}
				\item{The  support of the non-trivial vanishing cycle is a torus  bounding a Reeb component.}
			\end{enumerate}
		\end{theorem}

		\subsection {Norms on cohomologies}
	
		\subsubsection{$L^2$ -norm on cohomologies and   harmonic forms}\label{harm}
		
		Let $M^3$ be a closed oriented Riemannian 3-Manifold. 
	Recall that 
	\begin{equation}\label{S^1}
		H^1(M^3;\Z)\cong [M^3,S^1] \ \text{(homotopy classes of mappings to the circle)}
	\end{equation} 
	and each cohomological class 
	\begin{equation}\label{a}
		a\in H^1(M^3;\Z)
	\end{equation} 
	 can be obtained as an  image of the generator $[S^1]^*\in H^1(S^1;\Z)\cong \Z$ under the homomorphism $f^*:H^1 (S^1;\Z)\to H^1(M^3;\Z)$ induced by  the  mapping $f: M^3\to S^1$  uniquely  defined up to homotopy. Recall that the group $H_2(M^3;\Z)\overset{PD}\cong H^1(M^3;\Z)$ does not contain a torsion and we can identify $H^1(M^3;\Z)$ with  the integer lattice $H^1(M^3;\Z)_{\R}\subset H^1(M^3;\R)$ and $H_2(M^3;\Z)$ with   the integer lattice $H_2(M^3;\Z)_{\R}\subset H_2(M^3;\R)$ respectively. Observe that  the Poincar\'e duality $H^1(M^3;\R)\overset{PD}\cong  H_2(M^3;\R)$ induces the Poincar\'e duality  of integer lattices  $H^1(M^3;\Z)_{\R}\overset{PD}\cong  H_2(M^3;\Z)_{\R}. $
	
	Let  us identify $S^1 $ with the unit  length circle $\R/\Z$ with the natural parameter $\theta$. If $f$ is  a smooth function, then the preimage $f^{-1}(\theta)$ of a regular value $\theta\in S^1$ is a smooth (not necessarily connected)  oriented    submanifold $M^2\subset M^3,$  representing the class $[M^2]\in H_2(M^3;\Z)_{\R}$ dual to   to the cohomology class   $a$ (see \eqref{a}).

Let us  use the isomorphism between  singular cohomologies with real coefficients and de Rham cohomologies.
	 Recall that each homotopy  class in $ [M^3,S^1]$ can be represented by the harmonic mapping (\cite{ES}). 	Let $$u:M^3\to S^1$$ be  a harmonic map representing  the nontrivial  class $$[u] \in [M^3,S^1]\cong H^1(M^3;\Z).$$  Observe that $$\alpha = u^* d\theta, \ \theta \in S^1$$ is a harmonic 1-form  on $M^3$ corresponding to the  integer lattice class $[u]\in H^1(M^3;\Z)_{\R}$.
	
	On the space of differential $k$ -forms $\Omega^k(M^3)$, $k\in \{0,1,2,3\}$,   one can introduce  $L^2$-norm:  
	\begin{equation}\label{hn} 
		||\alpha||_{L^2} =\sqrt{\int_{M^3}\alpha \wedge *\alpha} = \sqrt{\int_{M^3}|\alpha|^2},\end{equation}
	where $*$ denotes the Hodge star operator  and  $$|\alpha_p| = \sqrt{*(\alpha_p \wedge *\alpha_p)},  \ p\in M^3.$$
	Observe,  in 3-dimensional vector space $T_pM^3$ each  k-form  $\alpha_p$  is simple and   $|\alpha_p|$ coincides with the comass norm: 
	$$  |\alpha_p| = \max \alpha_p(e_1,\dots,e_k),$$
	where the maximum  is taken over all orthogonal  frames of vectors $(e_1,\dots,e_k)$ in $T_pM^3$. 
	
	We will also use the   $L^{\infty}$-norm on $\Omega^*(M^3)$, which is defined as follows:
	$$||\alpha||_{L^{\infty}}:=\max_{p\in M^3} |\alpha_p|.$$
	
	The norm \ref{hn} induces the $L^2$-norm on the de Rham cohomology of $M^3$ as follows. Let $a\in H^k(M^3;\R)$, then we set 
	
	$$ ||a||_{L^2}:= \inf_{\alpha} \{||\alpha||_{L^2}: \alpha \in \Omega^k(M^3) \  \text{is a smooth closed k-form representing }  a\}. $$
	
	From de Rham  - Hodge theory it follows  $||a||_{L^2} = ||\alpha||_{L^2} $, where $$\alpha \in \mathcal H^1(M^3)\subset \Omega^1(M^3)$$ is the unique harmonic form  $(d\alpha=\delta\alpha=0 )$ representing the class  $a\in H^k(M^3;\R)$. \footnote{Through  $\mathcal H^1(M^3)$ we denote the space of harmonic 1-forms on $M^3$.}

	Using Poincar\'e duality $H_i(M^3;\R)\overset {PD}\cong H^{3-i}(M^3;\R)$ we can introduce the $L^2$ - norm on $H_2(M^3;\R)$ setting $$||b||_{L^2} := ||PD(b)||_{L^2}, \ b\in H_i(M^3;\R). $$  
	On the other hand, the non-degenerate Kronecker pairing
	$$
	<,>:H^k(M^3;\R) \times H_k(M^3;\R) \to \R
	$$
	induced by integration of closed forms over cycles, allows us to define  the $L^2$- norm  $||\cdot||^*_{L^2}$  on $H_k(M^3;\R)\cong (H^k(M^3;\R))^*$  dual to the $L^2$ - norm  $|| \cdot||_{L^2}$ on $H^k(M^3;\R)$. 
	As was shown in \cite {BK}  
	$$
	PD: (H^i(M^3 ;\R), || \cdot||_{L^2}) \to (H_{3-i}(M^3; \R), || \cdot||^*_{L^2}) 
	$$
	is an isometry for $i=1,2$.

	Note,  that $$PD([\alpha\wedge \beta] )=  PD([\beta\wedge \alpha])=<[\alpha],PD([\beta])>= <[\beta],PD([\alpha])>,$$
	where $\alpha\in \Omega^1(M^3)$ and $\beta\in \Omega^2(M^3)$ are closed forms. Since the  set of  integer-directed rays   from   $0\in H^1(M^3;\R)$  are everywhere dense set in  $H^1(M^3;\R)$.    we have:
	
	\begin{equation}\label{*-norm}
		|| b ||_{L^2} = ||PD(b)||^*_{L^2} = \sup_{a\not=0}\frac {<a,PD(b)>}{||a||_{L^2}}= \sup_{[\Sigma]\not = 0}\frac {<b,[\Sigma]>}{||[\Sigma]||_{L^2}},
	\end{equation}   
	where  $b \in H^2(M^3,\mathbb R)$,  $a \in H^1(M^3,\Z)_{\R}$ and  $\Sigma$ is a  compact oriented surface embedded in $M^3$ such that  $PD(a)=[\Sigma]. $

	Let us recall the following    inequality   (see \cite[[7.1.13,7.1.17, 9.2.7,9.2.8]{Pet}).  If $\alpha$ is a harmonic 1-form on closed Riemannian manifold $M^n$,  then 

\begin{equation}\label{Pet}
||\alpha||_{L^{\infty}} \leq \Lambda_n(k,D) ||\alpha||_{2}.
\end{equation} 
Here $||\alpha||_{2} =\frac{||\alpha||_{L^{2}}}{\sqrt{Vol(M^n)}}$,   $D>0$ is the constant satisfying the inequality   $Diam(M^n)\leq D$  and  $k\leq 0$   is the constant satisfying the inequality $Ric(M^3)\geq (n-1)k$.

		C.B. Croke   in \cite{Croke}   gave an estimate for the diameter of a closed  Riemannian manifold,  which we  adapt to 3-dimensional case:
		$$
		Diam (M^3)\leq  \frac{27\pi Vol(M^3)}{8\ inj(M^3)^{2}}.
		$$  
		In particular, if $M^3$ satisfies the conditions of  Theorem \ref{Main} we can take  $$D=\frac{27}{8}\pi\frac{ V_0}{ i_0^{2}}.$$
		Moreover, we can put $k=k_0$ \footnote {$k_0<0$ by the condition of Theorem \ref{Main}. }  and thus,   $ \Lambda_3(k,D) = \Lambda(V_0, i_0,k_0)$,  and we can rewrite the inequality \eqref{Pet} in our case as follows:

				\begin{equation}\label{Pet3}
				||\alpha||_{L^{\infty}}\leq  \frac{\Lambda(V_0, i_0,k_0)}{\sqrt{Vol(M^3)}}||\alpha||_{L^2}, \ \alpha\in \mathcal H^1(M^3).
							\end{equation}

		 In the case where $M^3$ is hyperbolic,  F. Brock and Nathan M. Dunfield \cite{BD} proved   the following inequality  to be true:
		
	\begin{equation}\label{Dunf2}
			||\alpha||_{L^{\infty}}\leq  \frac{5}{\sqrt{inj(M^3)}}||\alpha||_{L^2}, \ \alpha\in \mathcal H^1(M^3).
	\end{equation}

		\subsubsection{Thurston norm}\label{3.2.}

		The Thurston norm on $H_2(M^3;\mathbb Z)$ is defined in \cite{Th} as follows: 
		$$
		|| a ||_{Th} = inf \{\chi_{-}(\Sigma) | \Sigma \text {\ is an embedded  surface representing}   
		\	a\in H_2(M^3;\mathbb Z ) \},
		$$
		where $\chi_{-} (\Sigma) = max\{ -\chi(\Sigma), 0\}$. Recall that $\chi(\Sigma)=2-2g$ denotes the Euler characteristic of a surface $\Sigma$ of genus $g$.  When $\Sigma$ is not connected, define $\chi_- (\Sigma) $ to be the sum $\chi_- (\Sigma_1 ) + \dots +\chi_- (\Sigma_k)$, where $\Sigma_i, \ i=1,\dots, k$ are the connected components of $\Sigma$.
		As Thurston showed, the Thurston norm can be extended in a unique way to the seminorm in $H_2(M^3,\mathbb R)$. 
		
		Suppose that $M^3$ is irreducible atoroidal oriented 3-Manifold, for example, $M^3$ is   a closed oriented  3-Manifold of negative curvature, then   $|| \cdot ||_{Th} $  is a norm.  In this case we can define the dual Thurston norm as follows:

		\begin{equation}
			||\alpha||^*_{Th} = \sup_{\Sigma} \frac{<\alpha, [\Sigma]> }{||[\Sigma]||_{Th}}, 
		\end{equation}
		where  $\alpha \in H^2(M^3,\mathbb R)$ and the supremum being taken over all compact oriented surfaces $\Sigma$ embedded in $M^3$ such that $0\not=[\Sigma]\in H_2(M^3;\Z)$.  
		
		Thurston proved that the convex hull of the Euler classes of taut  foliations on $M^3$ is the unit ball for the dual Thurston norm. In particular,  the dual Thurston norm $||e(T\F)||^*_{Th}\leq 1$ for the taut foliation $\mathcal F$. 
		
		Define the Thurston norm on $H^1(M^3;\R)$ by setting 
		$$||a||_{Th} := ||PD(a)||_{Th}, \   a\in H^1(M^3;\R).$$

			Let $u:M^3\to S^1$ be  a harmonic map representing  the nontrivial  class $[u] \in [M^3,S^1]\cong H^1(M^3 ;\Z)\overset{PD}\cong H_2(M^3;\Z)$.

			  In  \cite{St}  D. Stern estimates an average Euler characteristic of a surface dual to the harmonic mapping of $M^3$  into the circle, which made it possible to estimate Thurston's norm from above:

			\begin{equation} \label{St1}
				||\alpha||_{Th}\leq \int_{\theta\in S^1}-\chi(\Sigma_{\theta}) \leq \frac{1}{4\pi}||\alpha||_{L^2}||R^-||_{L^2},
			\end{equation}
			where $R^-:=\min\{0,R\}$ is a negative part of the scaler curvature $R$ and $\alpha = u^*d\theta$.
	
	\begin{remark}
	A similar estimate was obtained by P. Kronheimer and T.   Mrowka  in \cite{KM} for the dual Thurston norm on the spce of 2-dimensional cohomologies   $H^2(M^3;\R)$ of irreducible atoroidal 3-Manifold $M^3$: 
			\begin{equation} \label{St1}
			||e||^*_{Th}\geq  {4\pi}\frac {||e||_{L^2}}{||R||_{L^2}}, \  e\in H^2(M^3;\R).
		\end{equation}
		
Moreover,  	they proved, that 
	
		\begin{equation} \label{St1}
		||e||^*_{Th} = \sup_{g\in {\mathcal G}} {4\pi}\frac {||e||_{L^2}}{||R||_{L^2}},
	\end{equation}
	where $\mathcal G$ is a space of all Riemannian metrics on $M^3$.  

P. Kronheimer and T. Mrowka  also showed  that in the case when $e\in H^2(M^3;\R)$ is a monopole class  arising  in 3-dimensional Siberg-Witten theory,  then $$  {4\pi}\frac {||e||_{L^2}}{||R||_{L^2}}\leq 1 $$  for all Riemannian metrics on $M^3$. In particular, $||e||^*_{Th} \leq 1$.
\end{remark}

	In the case, where $M^3$ is hyperbolic,  the  relation between  the Thurston's norm and the  $L^2$-norm on $H^1(M^3;\R)$ was obtained by F. Brock and Nathan M. Dunfield \cite{BD} :

			\begin{equation}\label{Dunf1}
				\frac{\pi}{\sqrt{Vol(M^3)}}||*||_{Th}\leq ||*||_{L^2}\leq  \frac{10\pi}{\sqrt{inj(M^3)}}||*||_{Th}.
			\end{equation}

	\section {Proof of  Main Theorem}

			{\it Proof of Part a).}  According to \eqref{Chi1}, it is enough to prove that  $\F$  is  taut  if $H_0\leq \sqrt {-K_0}$. 
			
			 Let us suppose  $\F$  is not taut, then  it contains a generalized Reeb component $A$, which is bounded by tori (see section \ref{Int}).  Since  a  closed negative curvature 3-manifold  has a hyperbolic fundamenal group, it does not contain $\Z^2$ as a subgroup and hence does not contain  an  incompressible torus.  Thus,   by Theorem \ref{Nov} $\F$ contain a Reeb component $\mathcal R$.  
			
			 Let us consider   the universal covering $p: \widetilde M^3\to M^3$ with the pull back Riemannian metric on $ \widetilde M^3$. By the Cartan - Hadamard theorem, the exponential map $$\exp: T_x\widetilde M^3\to \widetilde M^3$$  is a diffeomorphism.
			
		Consider the following commutative diagram describing the embedding $i: \mathcal R \to M^3$ and its lift to the universal coverings:
			$$
			\begin{CD}
				\widetilde {\mathcal R} @>\widetilde i>>\widetilde{M}^3 \\
            @VVp V    @VVpV \\
			  {\mathcal R} @>i>> M^3 
			\end{CD} $$
			
			Let us denote $\bar {\mathcal R}:= \widetilde i(\widetilde {\mathcal R})$.
			
			{\it Case 1.  The lift $\bar { {\mathcal R}}$ of $\mathcal R$ to  $\widetilde{M}^3$  is homeomorphic  to the  solid torus  $D^2\times S^1$.  }

			In this case  we can find a ball $D(r)\subset \widetilde{M}^3$ of radius 
			$r$  containing $\bar {R}$ such that it's bounding sphere $S(r)=\partial B(r)$ touches the torus $T^2=\partial \bar {\mathcal R}$ at some point $q$. Clearly that $S(r)$ is a supporting sphere for $T^2$  at $q$. By  \eqref{ineq1} we have: $ H_r(q)\leq H(q)$, where     $ H_r(q)$ and $ H(q)$ are mean curvatures of $S(r)$ and $T^2$ at the point $q$. But $\forall  r>0 $ we have $$H_r(q) \geq  H^0_r=\sqrt{-K_0}\coth (r\sqrt{-K_0}) >\sqrt{-K_0},$$ 
			where $H^0_r$ is the mean curvature of the round sphere $S(r)$ of radius $r$ in the space of constant curvature $K_0$ (see   \eqref{ineq1}).
			
			This contradicts to the assumption on $H_0$ in the Part $a)$. 
			\begin{figure}
		{\includegraphics[scale=0.5]{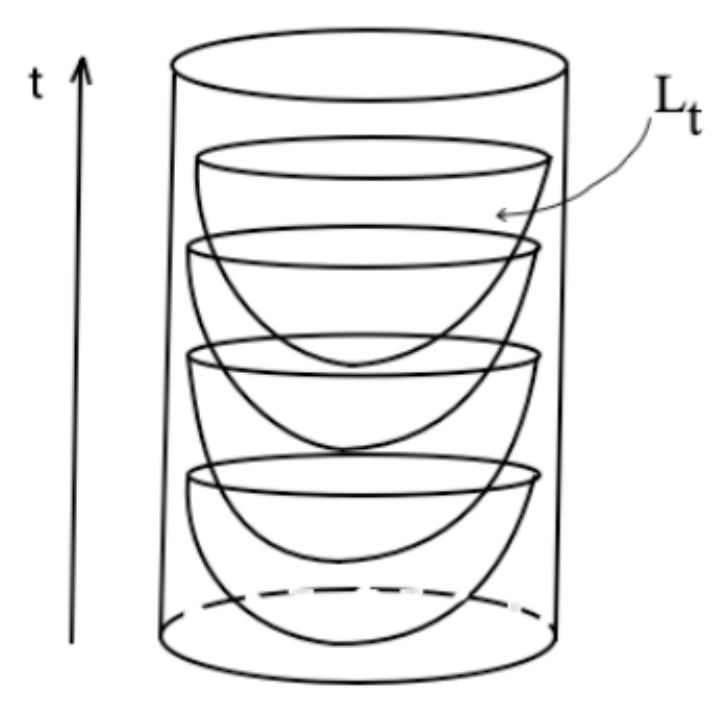}}
				\caption{ The universal covering $\widetilde {\mathcal R} $ }
		\label{Reeb3}
			\end{figure}
			
			{\it Case 2.  The lift $\bar { {\mathcal R}}$ of $\mathcal R$ to $\widetilde{M}^3$ is homeomorphic  to  the cylinder $D^2\times \R$.}

			Let us consider the disk $D\subset \bar {\mathcal R}$  corresponding  to the disk $D^2\times 0\subset D^2\times \R$  by the homeomorphism $\bar {\mathcal R}\simeq D^2\times \R$.   Consider the ball $B(r)\subset \widetilde{M}^3$ such that  $D\subset B(r)$ and its boundary   $S(r)=\partial B(r)$ is in general position with respect to the $ \partial \bar {\mathcal R}$. In this case the intersection $S(r)\cap \partial \bar {\mathcal R}$ is the family of circles ${\it C}= \{C_i\}$. Let  $A$ be a connected component  of the compact $B(r)\cap \bar {\mathcal R}$ containing $D$.  The intersection  $\partial A\cap S(r)$  consists of the disjoint family of compact connected surfaces ${\mathcal S}= \{S_i\} $  each of which has a boundary consisting of  circles of the family $\it C$. 	Since $B(r)\cap \bar{\mathcal R}$ is compact  there are two surfaces $\{S_1, S_2\}	\in \mathcal S$ belonging to the different connected components of $\bar {\mathcal R} \setminus D$.  
			
			Observe that the lifted foliation of  $$\Int \bar {\mathcal R}\cong \R^2\times \R$$ is  a direct product foliation of  $\R^2\times \R$  by  planes $L_t\simeq \R^2\times t$. 			
			Without loss of generality, we assume that a leaf $L_t$ goes to infinity if $t\to +\infty$ (see Fig. \ref{Reeb3}).  Then we can find the parameters $t_1,t_2$  such that   $L_{t_i}\cap S_i \not=\emptyset$, $i=1,2,$    and  $L_{t}\cap S_i =\emptyset$ for $t>t_i$,  $i=1,2.$   It is easy to show that each of the  surfaces $S_i$ divides the  $\bar {\mathcal R} $  into two connected components (see \cite{B1}). From the definition  of $S_1, S_2$   we conclude that at least one of the leaves $L_{t_i}, \ i=1,2,$ for instance  $L_{t_1}$,  must  intersect $D$.   But it means that  $L_{t_1}$ belongs to the same side of $S_1$ as $D$. It means that $S(r)$ is a supporting sphere for $L_{t_1}$ at the touch points $S_1\cap L_{t_1}$. The same reasoning as in Case $1 $ proves that the mean curvature of the leaf $L_{t_1}$ should be more $\sqrt{-K_0}$  at the touch points,  which leads to a contradiction.  This completes the proof of Part $a)$.
			\newline
			
			{\it Proof of Part b).} 
			 Let $0\not =\alpha \in \mathcal H^1(M^3)$. Pick a surface $\Sigma$  dual to 
			$[\alpha]\in H^1(M^3,\R)$ which is incompressible and realizes the Thurston norm, i.e.  $||[\alpha]||_{Th}=-\chi(\Sigma)$. Since $\Sigma$ is incompressible, by \cite {FHS}  we can assume that $\Sigma$ has least area in its isotopy class and hence is a stable minimal surface. 
			
			\begin{lemma} 
				\begin{equation} \label{9} Area (\Sigma)\leq  \frac{2\pi \chi(\Sigma)}{K_0}. \end{equation}
				\end{lemma}
			
			\begin{proof}
				Since $K_{\Sigma}\leq K|_{\Sigma}$ for a minimal surface, where $K_{\Sigma}$ denotes  the Gauss curvature of ${\Sigma}$,    the Gauss-Bonnet theorem yields $$-K_0 Area (\Sigma)\leq \int_{\Sigma}-K_{\Sigma}=-2\pi\chi(\Sigma).$$
						
				\end{proof}

			\begin{lemma}\label{10}
			$	||\alpha||_{L^2} \leq  -	\frac{2\pi\Lambda} {K_0\sqrt{Vol (M^3)}} ||\alpha||_{Th}.$
				\end{lemma}
				
				\begin{proof} 
					$$ 
					||\alpha||^2_{L^2} =\int_{M^3} \alpha\wedge*\alpha\overset{PD}=\int_{\Sigma}*\alpha\leq \int_{\Sigma}|*\alpha|dA = $$$$=\int_S|\alpha|dA \leq  \int_{\Sigma}||\alpha||_{\infty}dA \leq ||\alpha||_{\infty}Area(\Sigma) \overset{\eqref {9}+ \eqref{Pet3}}\leq 	\frac{2\pi\chi(\Sigma)\Lambda} {K_0\sqrt{Vol (M^3)}}||\alpha||_{L^2}. 
					$$
					
					Recalling that   $||[\alpha]||_{Th}:=||[\Sigma]||_{Th}=-\chi(\Sigma)$ the result follows. 
					\end{proof}
			Taking $V_0=Vol(M^3)$,  from \eqref{euler1} and Lemma \ref{10} we obtain:
			
			$$||e(T\F)||_{Th}= \sup_{[\Sigma]\not =0}\frac{<e(T\F), [\Sigma]>}{||[\Sigma ]||_{Th}}\leq \sup_{[\Sigma]\not =0} -	\frac{2\pi\Lambda}  {K_0\sqrt{Vol(M^3)}}\frac{<e(T\F), [\Sigma]>}{||[\Sigma ]||_{L^2}}\leq$$$$ \leq 6\frac{k_0}{K_0}\Lambda   -   \frac{64\pi H^2_0Vol(M^3)} {3K_0C^3_0}\Lambda^2.$$ 
			Now \eqref{euler3} follows for arbitrary $V_0\geq Vol(M^3)$. This proves Part b) of Theorem \ref{Main}.
			\newline 
			
			{\it Proof of Part c).}	
			
			 In \cite{B2}  the following estimate was obtained.
			
				\begin{equation}\label{euler}
				|e(T\F)([\Sigma])|  \leq \frac{1}{2\pi}||\alpha||_{L^2}||R^-||_{L^2}  +   \frac{32 H^2_0V_0^2} {3C^3_0}|| \alpha||_{L^{\infty}} . 
			\end{equation}

			Applying  \eqref{Dunf1} and \eqref{Dunf2} to \eqref{euler} and, considering that $R^-=-6$, we obtain:
			$$
			||e(T\F)||_{Th}  \leq 30\frac{\sqrt{Vol(M)}}{\sqrt{inj ( M)}}  +  \frac{ H^2_0V_0^2} {3C^3_0}\frac{1600\pi}{inj(M)}. 
			$$
			that yields part c) of Theorem \ref{Main}.



\begin{thebibliography}{10}
	
	\bibitem{B2}
	Dmitry Bolotov.
	\newblock The $l^2$ -norm of the euler class for foliations on closed
	irreducible riemannian 3-manifolds.
	\newblock URL: \url{https://arxiv.org/abs/2212.06807}.
	
	\bibitem{B1}
	Dmitry Bolotov.
	\newblock Foliations on closed three-dimensional riemannian manifolds with a
	small modulus of mean curvature of the leaves.
	\newblock {\em Izv. RAN. Ser. Mat.}, 86:85--102, 2022.
	
	\bibitem{B11}
	Dmitry Bolotov.
	\newblock On foliations of bounded mean curvature on closed three-dimensional
	riemannian manifolds.
	\newblock {\em Proceedings of the International Geometry Center},
	16(2):173--182, 2023.
	
	\bibitem{BD}
	Dunfield Nathan~M. Brock, Jeffrey~F.
	\newblock Norms on the cohomology of hyperbolic 3-manifolds.
	\newblock {\em Invent. math.}, 210:531–558, 2017.
	
	\bibitem{Croke}
	Christopher~B. Croke.
	\newblock Some isoperimetric inequalities and eigenvalue estimates.
	\newblock {\em Annales scientifiques de l'\'Ecole Normale Sup\'erieure},
	13(4):419--435, 1980.
	
	\bibitem{BZ}
	Victor~Zalgaller Dmitry~Burago.
	\newblock {\em Introduction to Riemannian Geometry [in Russian]}.
	\newblock Nauka. St. Petersburg, 1994.
	
	\bibitem{ES}
	James  Eells and  J.H. Sampson.
	\newblock Harmonic mappings of riemannian manifolds.
	\newblock {\em Am. J. Math.}, 86:621--657, 1964.
	
	\bibitem{FHS}
	Joel Scott, ~Peter Freedman, Michael~Hass.
	\newblock Least area incompressible surfaces in 3-manifolds.
	\newblock {\em Invent. Math.}, 71:609--642, 1983.
	
	\bibitem{Good}
	Sue~E Goodman.
	\newblock Closed leaves in foliations of codimension one.
	\newblock {\em Comm. Math. Helv.}, 50:383--388, 1975.
	
	\bibitem{KM}
	Peter Kronheimer and Tomasz Mrowka.
	\newblock Scalar curvature and the thurston norm.
	\newblock {\em Mathematical Research Letters}, 4:931–937, 1997.
	
	\bibitem{Nov}
	Sergej~Petrovich Novikov.
	\newblock Topology of foliations.
	\newblock {\em Trans. Moscow. Math. Soc.}, 14:268--304, 1967.
	
	\bibitem{Pet}
	Peter Petersen.
	\newblock {\em Riemannian Geometry}.
	\newblock Graduate Text in Mathematics, 3rd Edition. Springer, 2016.
	
	\bibitem{R}
	Geprges Reeb.
	\newblock Sur certaines proprietes topologiques des var\'i\'et\'es feuilletees.
	\newblock {\em Acrual Sci. Ind., Hermann, Paris}, 1183, 1952.
	
	\bibitem{Ros}
	Herold Rosenberg.
	\newblock Foliations by planes.
	\newblock {\em Topology}, 7:131--138, 1968.
	
	\bibitem{St}
	Daniel~L. Stern.
	\newblock Scalar curvature and harmonic maps to $s^1$.
	\newblock {\em J. Differential Geom.}, 122(2):259--269, 2022.
	
	\bibitem{Sul}
	Dennis Sullivan.
	\newblock A homological characterization of foliations consisting of minimal
	surfaces.
	\newblock {\em Comm. Math. Helv.}, 54:218--223, 1979.
	
	\bibitem{T}
	Itiro Tamura.
	\newblock {\em Topology of foliations: an introduction}, volume~97 of {\em
		Translation of Mathematical Monographs}.
	\newblock American Mathematical Society, Providence, RI, 1992.
	
	\bibitem{Th}
	William Thurston.
	\newblock A norm for the homology of 3-manifolds.
	\newblock {\em Memoirs of the American Mathematical Society}, 59(339):99--130,
	1986.
	
\end{thebibliography}



\end{document}